\documentclass[11pt,a4paper,reqno]{amsart}
\usepackage[utf8]{inputenc}
\usepackage[english]{babel}

\usepackage{lipsum}
\usepackage{amsfonts}
\usepackage{amsmath}
\usepackage{amssymb}
\usepackage{geometry}
\usepackage{amsthm}
\usepackage[dvipsnames]{xcolor}
\usepackage{esint}
\usepackage{tcolorbox}
\definecolor{mygray}{gray}{0.8}
\usepackage{cancel}

\newtheorem{theorem}{Theorem}[section]
\newtheorem{lemma}[theorem]{Lemma}
\newtheorem{corollary}[theorem]{Corollary}
\newtheorem{proposition}[theorem]{Proposition}
\theoremstyle{definition}
\newtheorem{definition}[theorem]{Definition}

\usepackage{hyperref}
\hypersetup{
    colorlinks=true,
    linkcolor=WildStrawberry,
    filecolor=cyan,      
    urlcolor=Cerulean,
    citecolor=Salmon,
    }

\begin{document}


\title[Weakly porous sets and $A_1$ weights in spaces of homogeneous type]{Weakly porous sets and $A_1$ Muckenhoupt weights in spaces of homogeneous type}

\author[Hugo Aimar]{Hugo Aimar}
\author[Ivana G\'{o}mez]{Ivana G\'{o}mez}
\author[Ignacio G\'{o}mez Vargas]{Ignacio G\'{o}mez Vargas}
\begin{abstract}
   In this work we characterize the sets $E\subset X$ for which there is some $\alpha>0$ such that the function $d(\cdot,E)^{-\alpha}$ belongs to the Muckenhoupt class $A_1(X,d,\mu)$, where $(X,d,\mu)$ is a space of homogeneous type, extending a recent result obtained by Carlos Mudarra in metric spaces endowed with doubling measures. In particular, generalizations of the notions of weakly porous sets and doubling of the maximal hole function are given and it is shown that these concepts have a natural connection with the $A_1$ condition of some negative power of its distance function. The proof presented here is based on Whitney-type covering lemmas built on balls of a particular quasi-distance equivalent to the initial quasi-distance $d$ and provided by Roberto Macías and Carlos Segovia in ``A well-behaved quasi-distance for spaces of homogeneous type'', Trabajos de Matemática 32, Instituto Argentino de Matemática, 1981, 1-18.
\end{abstract}

\maketitle

\section{Introduction}\label{sec:Introduction}

Sometimes, the extension of a particular result of real or harmonic analysis from Euclidean space to spaces of homogeneous type looks like a test for the robustness of the analytic techniques involved. Sometimes these extensions are useful to deal with models provided by problems in PDE. See for example \cite{SEVILLA} for a non-metric quasi-metric space of homogeneous type associated with the analysis of Harnack's inequality for degenerate parabolic equations. Nevertheless, for some analytical problems, the above are not the only reasons to focus on spaces of homogeneous type. Perhaps the most important reason is the fact that for a given quasi-distance $d$ on $X$, any symmetric function $\rho$ on $X\times X$ such that $c_1\rho\leq d\leq c_2\rho$, for some constants $c_1$ and $c_2$, is also a quasi-distance on $X$. This closure property under equivalence of quasi-distances, which is certainly not shared by the family of metrics on $X$, is particularly useful when the analytical problem under consideration is invariant by the change of equivalent quasi-distances. And even more convenient when some quasi-distance, equivalent to the original in the space, can be obtained with some specific required geometric property.

In the context of general spaces of homogeneous type, the classes of Muckenhoupt weights $A_p$ are widely studied for their various applications in singular integral theory, harmonic analysis, PDE, and other related topics. Classical examples of non-trivial $A_p$ weights in $\mathbb{R}^n$ consist of powers of distance functions of the sort $|x|^\alpha$, for $-n<\alpha<n(p-1)$ and $1<p<\infty$. Motivated by the study of regularity of PDE solutions with adequate boundary conditions, in \cite{AIMAR,LOPEZ}, sufficient requirements on a real number $\alpha$ and a closed subset $E$ of an Ahlfors metric measure space $(X,d,\mu)$ were found in order to satisfy the condition $d(x,E)^\alpha\in A_1(X,d,\mu)$.

More recently, in the articles \cite{ANDERSON,MUDARRA} by T. Anderson, J. Lehrbäck, C. Mudarra, and A. Vähäkangas, a complete characterization of subsets $E\subset X$ for which $d(x,E)^{-\alpha}$ belongs to the $A_1$ class for some $\alpha>0$ has been found, first in $\mathbb{R}^n$, and then for measure metric spaces $(X,d,\mu)$ satisfying the doubling condition. One of the necessary conditions on the set $E$ for this to be the case, besides some doubling condition on the maximal hole function $\rho_{d,E}$ to be introduced in the next pages, was therein termed as \textit{weak porosity}. Roughly speaking, a weakly porous set $E\subset X$ is a set such that all $d$-balls $B$ in $X$ contain a finite quantity of sub-balls $B_1,\hdots,B_N$ which do not intersect $E$ and whose measures sum at least a fixed proportion of the number $\mu(B)$. Intuitively, one can think of $E$ as a set that is ``full of pores everywhere'' with respect to both the quasi-distance $d$ and the measure $\mu$ on the underlying space, in contrast with other more common definitions of porosity which relate only to the metric $d$ and have no relation at all with any existing measure \cite{FALCONER,SHMERKIN}. For some previous results related to the problem considered here, see \cite{VASIN,DYDA}.

The proof of such equivalence in the Euclidean case provided in \cite{ANDERSON} strongly relies on the classical dyadic partitions of cubes in $\mathbb{R}^n$, and was generalized to the case of a measure metric space with a doubling measure in \cite{MUDARRA} by the use of the ``dyadic cubes'' due to Christ \cite{CHRIST}, as well as some refinements made later on (see \cite{HYTONEN}). In the present work, we intend to extend the same result obtained in \cite{MUDARRA} to spaces of homogeneous type $(X,d,\mu)$.
The rich theory of spaces of homogeneous type provides us with a broad set of tools to better face the problem at hand. In particular, we shall make use of the main result in \cite{MACIAS1} and a Whitney-type covering lemma in \cite{MACIAS2}.

The main result to be proved in this paper is the equivalence of weak porosity plus the doubling condition of the so-called ``maximal hole function'' for a set $E$, introduced by C. Mudarra in metric spaces,   to the $A_1$ condition for some negative power of $d(\cdot,E)$ in general spaces of homogeneous type. Its statement is presented next, while the precise definitions involved are given in the upcoming sections.

\begin{theorem}
    \label{teo1}
    Let $(X,d,\mu)$ be a space of homogeneous type such that every $d$-ball is an open set. Let $E$ be a non-empty subset of $X$. Then, the following statements are equivalent,
    \begin{itemize}
        \item[(I)] $E$ is weakly porous and $\rho_{d,E}$ is doubling;
        \item[(II)] There is some $\alpha>0$ such that $d(\cdot,E)^{-\alpha}\in A_1(X,d,\mu)$.
    \end{itemize}
\end{theorem}

Let us recall that in \cite{MACIAS3}, an equivalent quasi-distance in any quasi-metric space is built in such a way that the balls are open sets.

This paper is organized as follows. Section~\ref{sec:basicAnalyticGeometricAspects} is devoted to introduce the basic definitions and results that we shall use later in the paper, including the quasi-distance built by Macías and Segovia, the Whitney type covering lemma, and the definition of $A_1$ Muckenhoupt weights. In Section~\ref{sec:weaklyPorousSets} we introduce the generalization to spaces of homogeneous type of the notions of weak porosity and the doubling of the maximal hole function and we prove their invariance under changes of equivalent quasi-distances. The proof that the $A_1$-Muckenhoupt condition of $d(\cdot,E)^{-\alpha}$ for some positive $\alpha$ implies the doubling property of the maximal hole function and the weak porosity of $E$ is given in Section~\ref{sec:proofA1impliesweakporosityAndDoubling}. Finally, Section~\ref{sec:proofWeakPorosityAndDoublingimpliesA1} is devoted to the proof of the sufficiency of the weak porosity and the doubling of the maximal hole function on a set $E$ to get that $d(\cdot,E)^{-\alpha}\in A_1(X,d,\mu)$ for some $\alpha>0$.

\section{Basic analytic and geometric aspects of spaces of homogeneous type}\label{sec:basicAnalyticGeometricAspects}

Given a non-empty set $X$, we say that a function $d:X\times X\to[0,\infty)$ is a quasi-distance in $X$ if there is some constant $K$ such that for every $x,y,z\in X$ we have
\begin{itemize}
    \item[\textit{i.}] $d(x,y)=0$ 
    if and only if $x=y$;
    \item[\textit{ii.}] $d(x,y)=d(y,x)$;
    \item[\textit{iii.}] $d(x,z)\leq K[d(x,y)+d(y,z)]$.
\end{itemize}

Assuming such a constant $K$ exists for a set $X$ and a function $d$ as above, $K$ is necessarily greater than or equal to one, as long as $X$ is not a one-point set, and is called a triangular constant for $d$. For any quasi-distance $d$, we denote
$$K_d:=\min\{K\geq1:d\text{ satisfies property \textit{iii} with constant $K$}\},$$
so $K_d$ is then the least possible triangular constant for $d$. If $K_d=1$, we say as usual that $d$ is a metric. In any case, we refer to the pair $(X,d)$ as a quasi-metric space. Notice that if $Y\subset X$ with $Y\neq \emptyset$, then the restriction $d^*:={d|}_{Y\times Y}$ is a quasi-distance on $Y$ with $K_{d^*}\leq K_d$. Any quasi-distance $d$ on a set $X$ induces a topology $\tau_d$ via the uniform structure $\mathcal{U}_d=\{U\in\mathcal{P}(X\times X):\text{ there exists } r>0\text{ with }B_d(r)\subset U\}$, where $B_d(r)=\{(x,y)\in X\times X:d(x,y)<r\}$. The resulting topology is given by
$$\tau_d=\{A\subset X:\text{for every }x\in A,\text{ there exists } r>0\text{ such that }B_d(x,r)\subset A\},$$
where $B_d(x,r)=\{y\in X:d(x,y)<r\}$ denotes the $d$-ball with center $x$ and radius $r$. As it is well known, the $d$-balls do not need to be open sets for general quasi-metric spaces.

It is worthy to mention at this point that the uniformity approach to the topology defined in $X$ by a quasi-distance $d$ is used in \cite{MACIAS3} to obtain a distance $\rho$ such that $d$ and $\Tilde{d}:=\rho^s$ are equivalent, for some $s\geq 1$. The main tool is a lemma of metrization of uniform spaces with countable bases due to Aline Huke-Frink \cite{FRINK}. See also \cite{GUSTAVSSON}. In particular, the $\Tilde{d}$-balls are open sets.


In the current context of $(X,d)$ a quasi-metric space for which the $d$-balls are open, a space of homogeneous type is a triple $(X,d,\mu)$, where $\mu$ is a Borel measure such that for some constant $A>0$ the inequalities
$$0<\mu(B_d(x,2r))\leq A\:\mu(B_d(x,r))<\infty,$$
hold for every $x\in X$ and $r>0$.

Consider now the space $L^1_\text{loc}(X,d,\mu)$ defined as the set of all measurable functions $\phi:X\to\mathbb{R}$ such that $\int_{B_d(x,r)}|\phi|d\mu<\infty$ for every $x\in X$ and every $r>0$. We say that $w$ is a weight in $X$ if  $w\in L^1_\text{loc}(X,d,\mu)$ and $w$ is nonnegative. The specific weights we are interested in are introduced in the following definition that generalizes the well-known Muckenhoupt classes in Euclidean settings.

\begin{definition}
    \label{def1}
    Given a weight $w$ in $X$, we will say that $w$ belongs to the Muckenhoupt class $A_1(X,d,\mu)$ if there is a constant $C>0$ such that
        \begin{equation}
            \frac{1}{\mu(B_d(x,r))}\int_{B_d(x,r)}w(y)d\mu(y) \leq C\:\text{ess inf}_{y\in B_d(x,r)}w(y),\:\forall x\in X,\:r>0.
        \end{equation}
\end{definition}

The best possible constant for which the above inequality is valid is denoted as $[w]_{A_1(X,d,\mu)}$, or simply $[w]_{A_1}$.

\begin{proposition}
    \label{prop1}
    Let $E\subset X$ be a non-empty set, then
    \begin{itemize}
        \item[\textit{(i)}] for every $\theta\in\mathbb{R}$, the function $d(\cdot,E)^\theta:X\to[0,\infty)$ is measurable. Here, $d(x,E)=\inf_{e\in E}d(x,e)$;
        \item[\textit{(ii)}] if $\Bar{E}$ denotes the closure of $E$, then $\Bar{E}=\{x\in X:d(x,E)=0\}$.
    \end{itemize}
\end{proposition}

\begin{proof}
    To prove \textit{(i)}, notice that when $\theta=0$, $d(\cdot,E)^\theta$ is constant. On the other hand, for $\theta\neq0$, the upper or the lower level sets of $d(\cdot,E)^\theta$ can be written as a union of balls. Since we are assuming that the $d$-balls are open sets, those level sets are open and hence measurable. To prove \textit{(ii)}, take $x\in X\setminus\bar{E}$, then there is some $r>0$ such that $B_d(x,r)\subset X\setminus\bar{E}$, which means $d(x,E)\geq r>0$. Consequently, $d(x,E)=0$ implies that $x\in\bar{E}$. Besides, if $x\in\bar{E}$, $B_d(x,r)\cap E\neq\emptyset$ for every $r>0$. It follows that, for each $r>0$, there is some $e_r\in E$ such that $d(x,e_r)<r$. Then, $d(x,E)=0$.
\end{proof}

Item \textit{(i)} of Proposition~\ref{prop1} gives us a little peace of mind in the sense that the functions $d(\cdot,E)^{-\alpha}$, which we want to characterize as elements of the class $A_1(X,d,\mu)$, are always measurable, whereas \textit{(ii)} will be used repeatedly in the course of this work.

A very common technique while working in spaces of homogeneous type is the construction of special quasi-distances, which are comparable to the original, to attack problems that would otherwise be difficult to solve. Motivated by this, consider the collection $\partial(X)=\{\Tilde{d}:X\times X\to[0,\infty)\text{ such that }\Tilde{d}\text{ is a quasi-distance in }X\}$. Note $\partial(X)\neq\emptyset$ whenever $X\neq\emptyset$ as the function $\Tilde{d}(x,y)=1$ if $x\neq y$ and $\Tilde{d}(x,x)=0$ is always a metric in $X$. We introduce an equivalence relation $\sim$ on $\partial(X)$ requiring that $d'\sim d''$ if and only if there are constants $0<c_1\leq c_2<\infty$ such that the inequalities
\begin{equation*}
    c_1 d'(x,y)\leq d''(x,y)\leq c_2d'(x,y),
\end{equation*}
hold for every $x$ and $y$ in $X$. It is not difficult to see that equivalent quasi-distances induce the same topologies over the set $X$. As previously discussed, we are interested in quasi-distances $\Tilde{d}$ equivalent to $d$, but we will also require that the balls associated with $\Tilde{d}$ are still measurable sets. Motivated by this, we consider the set
$$\partial(X,d)=\{\Tilde{d}\in\partial(X):\Tilde{d}\sim d\text{ and }B_{\Tilde{d}}(x,r)\text{ is open for every }x\in X,\:r>0\}.$$

Notice that any $\Tilde{d}\in\partial(X,d)$ admits a doubling constant $A_{\Tilde{d}}$ for the measure $\mu$ depending on $c_1$, $c_2$ and $A$. From now on, we write $A_d=A$ to differentiate between doubling constants defined for other quasi-distances.

The next result provides estimates for the triangular constants in terms of the equivalence constants.

\begin{proposition}
    If $\Tilde{d}\sim d$ is a quasi-distance in $X$ and $0<c_1\leq c_2<\infty$ are constants such that
    \begin{equation}
        \label{eq5}
        c_1d(x,y)\leq\Tilde{d}(x,y)\leq c_2d(x,y),\:\forall x,y\in X,
    \end{equation}
    then
    \begin{equation}
        \label{eq6}
        \frac{c_1}{c_2}\leq \frac{K_{\Tilde{d}}}{K_d} \leq \frac{c_2}{c_1},
    \end{equation}
\end{proposition}

\begin{proof}
    Fixed some points $x,y,z\in X$, we have
    \begin{equation*}
        \Tilde{d}(x,z) \leq c_2 d(x,z) \leq c_2K_d[ d(x,y) + d(y,z) ] \leq \frac{c_2}{c_1} K_d [\Tilde{d}(x,y) + \Tilde{d}(y,z) ].
    \end{equation*}
    \noindent This says $K_{\Tilde{d}}\leq c_2c_1^{-1}K_d$. The remaining inequality follows by interchanging $d$ and $\Tilde{d}$ in the previous estimates.
\end{proof}

\begin{proposition}
    \label{propo2.4}
    Given $\alpha>0$ and $\Tilde{d}\in\partial(X,d)$, we have that $d(\cdot,E)^{-\alpha}\in A_1(X,d,\mu)$ if and only if $\Tilde{d}(\cdot,E)^{-\alpha}\in A_1(X,\Tilde{d},\mu),$
    for any non-empty set $E\subset X$.
\end{proposition}

\begin{proof}
    Notice that it is enough to prove one of the implications. Let $0<c_1\leq c_2<\infty$ be such that $c_1d(x,y)\leq\Tilde{d}(x,y)\leq c_2d(x,y)$ and suppose $d(\cdot,E)^{-\alpha}\in A_1(X,d,\mu)$. Given $y\in X$ and $r>0$, we have
    \begin{align*}
        \frac{1}{\mu(B_{\Tilde{d}}(y,r))}\int_{B_{\Tilde{d}}(y,r)}\Tilde{d}(x,E)^{-\alpha} & d\mu(x)\\ &\leq\frac{c_1^{-\alpha}}{\mu(B_{\Tilde{d}}(y,r))}\frac{\mu(B_d(y,c_1^{-1}r))}{\mu(B_d(y,c_1^{-1}r))}\int_{B_d(y,c_1^{-1}r)}d(x,E)^{-\alpha}d\mu(x) \\
        &\leq \frac{A_d^m}{c_1^\alpha}[d(\cdot,E)^{-\alpha}]_{A_1(X,d,\mu)}\text{inf ess}_{B_d(y,c_1^{-1}r)}d(\cdot,E)^{-\alpha} \\
        &\leq \frac{A_d^m}{c_1^\alpha}[d(\cdot,E)^{-\alpha}]_{A_1(X,d,\mu)}\text{inf ess}_{B_{\Tilde{d}}(y,r)}d(\cdot,E)^{-\alpha} \\
        &\leq A_d^m\Big(\frac{c_2}{c_1}\Big)^\alpha[d(\cdot,E)^{-\alpha}]_{A_1(X,d,\mu)}\text{inf ess}_{B_{\Tilde{d}}(y,r)}\Tilde{d}(\cdot,E)^{-\alpha},
    \end{align*}
    where in the second inequality we used the estimate
    $$\frac{\mu(B_d(y,c_1^{-1}r))}{\mu(B_{\Tilde{d}}(y,r))}\leq\frac{\mu(B_d(y,c_1^{-1}r))}{\mu(B_d(y,c_2^{-1}r))}\leq A_d^m$$
    with $2^{m-1}<\frac{c_2}{c_1}\leq2^m$.
\end{proof}

The particular quasi-distance to be used in Section~\ref{sec:proofWeakPorosityAndDoublingimpliesA1} to prove one of the implications of Theorem~\ref{teo1} is due to Macías and Segovia \cite{MACIAS1,AIMAR2}. As the next theorem states, such a quasi-distance can always be constructed for a given quasi-metric space and possesses some very good properties that will be of great help to prove our main result. Before continuing, let us recall the set operation of composition between relations on a set $X$. Given $W_1,W_2\subset X\times X$, their composition is the new set of $X\times X$ given by
$$W_1\circ W_2:=\{(x,z)\in X\times X:\exists y\in X\text{ such that }(x,y)\in W_2\text{ and }(y,z)\in W_1\}.$$

\begin{definition}
    \label{def4}
    Pick $0<a<(2K_d)^{-1}$ and set $B_d(r)=\{(x,y)\in X\times X:d(x,y)<r\}$, for each $r>0$,  $U(r,0)=B_d(r)$ and for each $n\in\mathbb{N}$ take $U(r,n)=B_d(a^nr)\circ U(r,n-1)\circ B_d(a^n r)$. Finally, define $V(r)=\bigcup_{n=0}^\infty U(r,n)$. The \textbf{Mac\'ias-Segovia quasi-distance} on $X$ induced by $d$ is the function $\delta:X\times X\to[0,\infty)$ given by
    $$\delta(x,y):=\inf\{r>0:(x,y)\in V(r)\}.$$
\end{definition}
The fact that $\delta$ is indeed a quasi-distance is contained in the next result.
\begin{theorem}
    \label{teo2}
    The function $\delta(\cdot,\cdot)$ satisfies the following properties.
    \begin{itemize}
        \item[\textit{(i)}] $\delta$ is a quasi-distance with $K_\delta\leq3K_d^3$. Furthermore, for every $x,y\in X$, we have
        $$\delta(x,y)\leq d(x,y)\leq 3K_d^2\delta(x,y).$$
        \item[\textit{(ii)}] For every $x\in X$ and $r>0$,
        $$B_\delta(x,r)=\{y\in X:(x,y)\in V(r)\}.$$
        Consequently, $\delta$-balls are open sets and $\delta\in\partial(X,d)$.
        \item[\textit{(iii)}] There exists a constant $\beta=\beta(K_d)\in(0,1)$ such that for every $x\in X$, $r>0$, $0<t\leq2K_\delta r$ and $y\in B_\delta(x,r)$ there is a point $z\in B_\delta(x,r)$ with
        \begin{equation}
            \label{eq:aux}
            B_\delta(z,\beta t)\subset B_\delta(y,t)\cap B_\delta(x,r).
        \end{equation}
        \item[\textit{(iv)}] The class $\{(B_\delta(x,r),\delta,\mu):x\in X,r>0\}$ is a uniform family of spaces of homogeneous type. More precisely, for every $x\in X$ and $r>0$, if $\delta^*$ and $\mu^*$ are the restrictions of $\delta$ and $\mu$ to $B_\delta=B_\delta(x,r)$, then the triplet $(B_\delta,\delta^*,\mu^*)$ is a space of homogeneous type admitting $K_\delta$ as a possible triangular constant as well as some doubling constant $A_{\delta^*}=A_{\delta^*}(A_d,K_d)$ both independent of $x$ and $r$.
    \end{itemize}
\end{theorem}

\begin{proof}
    Items (\textit{i}) and (\textit{iv}) as well as the fact that $B_\delta(x,r)=\{y\in X:(x,y)\in V(r)\}$ are merely restatements of Lemma~2.6 and Corollaries~2.5 and 2.11 of \cite{MACIAS1}. We now prove the remaining assertion in (\textit{ii}). Note that $\{y\in X:(x,y)\in V(r)\}=\bigcup_{n=0}^\infty\{y\in X:(x,y)\in U(r,n)\}$. Taking $W_n:=B_d(a^nr)\circ U(r,n-1)$ and letting $\Omega(x):=\{y\in X:(x,y)\in \Omega\}$ denote the $x$-slice of a set $\Omega\in\mathcal{P}(X\times X)$, we find
    \begin{align*}
        \{y\in X:(x,y)\in U(r,n)\} &= \{y\in X:(x,y)\in W_n\circ B_d(a^nr)\} \\
        &=\{y\in X:\text{there is some }w\in W_n(x)\text{ s.t. }(w,y)\in B_d(a^nr)\} \\
        &= \bigcup_{w\in W_n(x)}B_d(w,a^nr),
    \end{align*}
    which is open in $(X,d)$ and so is $[V(r)](x)=B_\delta(x,r)$, for any $r>0$. Because $\delta\sim d$ and $\delta$-balls are open sets, it follows that $\delta\in\partial(X,d)$.

    Finally, a statement like the one asserted in item (\textit{iii}) can not be found in \cite{MACIAS1}, but it can be deduced from the proof of Theorem 2.7 therein. Indeed, there it is shown that for every $x\in X$, $r>0$, $y\in B_\delta(x,r)$ and $0<t\leq2K_dr$ the inclusion
    $$B_\delta\Big(z,\frac{a^{3-p}}{3K_d^2}t\Big)\subset B_\delta(y,t)\cap B_\delta(x,r)$$
    holds, being $p$ the integer satisfying $a^p<2K_d\leq a^{p-1}$. This says that \eqref{eq:aux} is valid for $t\in(0,2K_dr]$ using $\beta_0:=(3K_d^2)^{-1}a^{3-p}$ instead of $\beta$. Since $2K_\delta r(3K_d^2)^{-1}\leq2K_dr$ by (\textit{i}), we can choose $\beta=(3K_d^2)^{-1}\beta_0=(3K_d^4)^{-1}a^{3-p}$ for \eqref{eq:aux} to be valid for every $t\in(0,2K_\delta r]$.
\end{proof}

Let us recall that equivalent quasi-distances induce the same topologies on the set $X$. So that the expression $A$ is an open set in $X$ is not ambiguous and we use it regardless of the quasi-distance.

To end this section, we state a basic Whitney-type covering lemma that follows directly from \cite[Lemma~2.9]{MACIAS2} and that will be of great help in Section~\ref{sec:proofWeakPorosityAndDoublingimpliesA1}.

\begin{lemma}
    \label{lemma0}
    Let $\Omega$ be an open, proper, and bounded subset of $Y$, for some space of homogeneous type $(Y,d',\nu)$, and suppose $K>0$ is any quasi-triangular constant for $d'$. Then, there exists a family of $d'$-balls $\{B_{d'}(x_i,r_i)\}_{i\in I}$, indexed by a countable set $I$, such that
    \begin{itemize}
        \item[\textit{(a)}] $B_{d'}(x_i,r_i)\cap B_{d'}(x_j,r_j)=\emptyset$, whenever $i\neq j$;
        \item[\textit{(b)}] $\bigcup_{i\in I}B_{d'}(x_i,4Kr_i)=\Omega$;
        \item[\textit{(c)}] $4K r_i\leq d'(x,X\setminus\Omega)\leq12K^3 r_i$, $\forall x\in B_{d'}(x_i,4K r_i)$, $i\in I$;
        \item[\textit{(d)}] For each 
        $i\in I$, there exists some $y_i\in Y\setminus\Omega$ such that $d'(x_i,y_i)<12K^2 r_i$.
    \end{itemize}
\end{lemma}

\section{Weakly porous sets and maximal holes}\label{sec:weaklyPorousSets}

The next two definitions introduce the concepts of maximal holes and weakly porous subsets in the setting of spaces of homogeneous type. Along this section, we shall assume as before that $(X,d,\mu)$ is a space of homogeneous type such that $d$-balls are open sets.

\begin{definition}
    \label{def2}
    Given a non-empty set $E\subset X$ and some ball $B_d(x,r)$, we consider the set $\Lambda(x,r;d,E)=\{0<s\leq2K_dr:\text{there exists } y\in X\text{ such that }B_d(y,s)\subset B_d(x,r)\setminus E\}$. The \textbf{maximal $E$-free hole function} is then defined as
    \begin{equation}
        \label{eq3}
        \rho_{d,E}(B_d(x,r)):=\sup\Lambda(x,r;d,E).
    \end{equation}
    \noindent In case $\Lambda(x,r;d,E)=\emptyset$, we set $\rho_{d,E}(B_d(x,r))=0$.
\end{definition}

The maximal $E$-free hole function gives the supremum over all radii of existing pores in a given ball $B_d(x,r)$, relative to the set $E$ and the quasi-distance $d$, but having a restriction in the pore size. This upper bound of $2K_dr$ prevents $\rho_{d,E}$ from reaching large or even infinite values and allows one to consider the following notion of doubling of this function with respect to the ball radius.

\begin{definition}
    \label{def2bis}
    Following \cite{MUDARRA}, we say that the function $\rho_{d,E}$ is \textbf{doubling} if there exists a constant $C_{d,E}>0$ such that
    \begin{equation}
        \label{eq4}
        \rho_{d,E}(B_d(x,2r))\leq C_{d,E}\:\rho_{d,E}(B_d(x,r)),\text{ for every $x\in X$ and $r>0$}.
    \end{equation}
\end{definition}

We note that the doubling property for $\rho_{d,E}$ can be seen as a uniform $\Delta_2$ condition for the family of functions of $r$, $\rho_{d,E}(B_d(x,r))$, with $x\in X$. Some basic results concerning this function are summarized in the two lemmas.

\begin{lemma}
    Fixed some $x\in X$ and given $E\in\mathcal{P}(X)\setminus\{\emptyset\}$, the function $\rho_{d,E}(B_d(x,\cdot)):(0,\infty)\to\mathbb[0,\infty)$ is nondecreasing.
\end{lemma}

\begin{proof}
    The statement follows directly from the Definition~\ref{def2}.
\end{proof}

\begin{lemma}
    \label{lemma3.3}
    Let $E$ be a non-empty set in $X$ and let $B=B_d(y,r)$ be some ball. If $\rho_{d,E}$ is doubling and $B\cap E\neq\emptyset$, then there exists a constant $C_0>0$ independent of $B$ such that
    $$d(x,E)\leq C_0\rho_{d,E}(B),\text{ for every }x\in B\setminus\Bar{E}.$$
\end{lemma}

\begin{proof}
    First, notice that if we were to have $\rho_{d,E}(B)=0$, this would mean that $d(x,E)=0$ for every $x\in B$. To see this start by taking such a point $x\in B$. Since balls are open in $X$, there is some $0<s_0\leq2K_dr$ such that $B_d(x,s)\subset B$ for every $s\leq s_0$. If $\rho_{d,E}(B)=0$, then $s\notin\Lambda(y,r;d,E)$ and there should exist $e_s\in E$ such that $e_s\in B_d(x,s)$ for each $0<s\leq s_0$, which would imply $d(x,E)=0$ for every $x\in B$. So in this case $B\setminus\Bar{E}=\emptyset$ and there is nothing to prove.

    If $\rho_{d,E}(B)>0$, we have that $B\setminus\Bar{E}\neq\emptyset$. Fixed some $x$ in $B\setminus\Bar{E}$ and setting $t=d(x,E)$, we have that $0<t\leq2K_dr$, where the upper bound follows from the fact that $B\cap E\neq\emptyset$. On the other hand, given $z\in B_d(x,t)$, we have $d(y,z)\leq K_d[d(y,x)+d(x,z)]
    \leq K_d(1+2K_d)r$. This means that $B_d(x,t)\subset B_d(y,K_d[1+2K_d]r)\setminus E$ and consequently $t\in\Lambda(x_0,K_d[1+2K_d]r;d,E)$. From this, we get
        \begin{align*}
            d(x,E)=t\leq\rho_{d,E}(B_d(y,K_d[1+2K_d]r))\leq C_{d,E}^m\rho_{d,E}(B),
        \end{align*}
        for every $x\in B\setminus\Bar{E}$, with $m$ such that $2^{m-1}<K_d(2K_d+1)\leq2^m$.
\end{proof}

Lemma~\ref{lemma3.3} shows that the values of $d(\cdot,E)$ over balls $B$ which intersect with the set $E$ are controlled by the maximal $E$-free hole function. The next results, in turn, establish the relationship between maximal $E$-free hole functions relative to equivalent quasi-distances.

\begin{lemma}
    \label{lemma3.5}
    Suppose $E\subset X$ is a non-empty set and $\Tilde{d}\sim d$ is a quasi-distance in $X$ with constants $0<c_1\leq c_2<\infty$ as in \eqref{eq5}. Then, for every $x\in X$ and $r>0$ we have
    \begin{equation}
        \label{eq7}
        c_1\rho_{d,E}(B_d(x,c_2^{-1}r))\leq\rho_{\Tilde{d},E}(B_{\Tilde{d}}(x,r)) \leq c_2\rho_{d,E}(B_d(x,c_1^{-1}r)).
    \end{equation}
\end{lemma}

\begin{proof}
    By the symmetry of the problem at hand, it suffices to prove the second inequality. Given $x\in X$ and $r>0$, we have
    \begin{align*}
        \rho_{\Tilde{d},E}&(B_{\Tilde{d}}(x,r)) \\&= \sup\Lambda(x,r;\Tilde{d},E) \\
        &= \sup\{0<s\leq 2K_{\Tilde{d}} r:\exists y\in X\text{ s.t. }B_{\Tilde{d}}(y,s)\subset B_{\Tilde{d}}(x,r)\setminus E\} \\
        &\leq \sup\{0<s\leq 2K_{\Tilde{d}} r:\exists y\in X\text{ s.t. }B_d(y,c_2^{-1}s)\subset B_d(x,c_1^{-1}r)\setminus E\} \\
        &= c_2 \sup\{0<s\leq 2c_2^{-1}K_{\Tilde{d}} r:\exists y\in X\text{ s.t. }B_d(y,s)\subset B_d(x,c_1^{-1}r)\setminus E\} \\
        &= c_2 \sup\{0<s\leq 2K_d(c_2^{-1}K_d^{-1}K_{\Tilde{d}} r):\exists y\in X\text{ s.t. }B_d(y,s)\subset B_d(x,c_1^{-1}r)\setminus E\} \\
        &\leq c_2\sup\{0<s\leq 2K_d(c_1^{-1} r):\exists y\in X\text{ s.t. }B_d(y,s)\subset B_d(x,c_1^{-1}r)\setminus E\} \\
        &= c_2\rho_{d,E}(B_d(x,c_1^{-1}r)),
    \end{align*}
    \noindent where the first inequality is justified by the inclusions $B_d(y,c_2^{-1}s)\subset B_{\Tilde{d}}(y,s)$ and $B_{\Tilde{d}}(x,r)\subset B_d(x,c_1^{-1}r)$. The third equality follows from rewriting the set
    $$\{0<s\leq 2K_{\Tilde{d}} r:\exists y\in X\text{ s.t. }B_d(y,c_2^{-1}s)\subset B_d(x,c_1^{-1}r)\setminus E\}$$
    \noindent as
    $$c_2\{0<t\leq 2K_{\Tilde{d}} c_2^{-1} r:\exists y\in X\text{ s.t. }B_d(y,t)\subset B_d(x,c_1^{-1}r)\setminus E\},$$
    where $aU$ denotes the dilation $\{au\in\mathbb{R}:u\in U\}$ of $U\subset\mathbb{R}$ by the number $a$. The last inequality follows from the fact that $2 K_d (c_2^{-1} K_{\Tilde{d}} K_d^{-1} r)\leq 2 K_d c_1^{-1} r$ (by \eqref{eq6}) implies the set
    $$\{0<s\leq 2K_d(c_2^{-1}K_d^{-1}K_{\Tilde{d}} r):\exists y\in X\text{ s.t. }B_d(y,s)\subset B_d(x,c_1^{-1}r)\setminus E\}$$
    is contained in $\Lambda(x,c_1^{-1}r;d,E)$.
\end{proof}

\begin{corollary}
    \label{coro3.5}
    Suppose $E\subset X$ is a non-empty set and let $\Tilde{d}\sim d$ be a quasi-distance in $X$. Then, $\rho_{d,E}$ is doubling if and only if $\rho_{\Tilde{d},E}$ is doubling.
\end{corollary}

\begin{proof}
    Take $0<c_1\leq c_2<\infty$ as in $\eqref{eq5}$ and let $m\geq1$ be the integer such that $2^{m-2}<\frac{c_2}{c_1}\leq2^{m-1}$. Assume now that $\rho_{d,E}$ is doubling so we can check the first implication. We have,
    \begin{equation*}
        \rho_{\Tilde{d},E}(B_{\Tilde{d}}(x,2r)) \leq c_2\rho_{d,E}(B_d(x,\tfrac{2r}{c_1})) \leq c_2 C_{d,E}^m\rho_{d,E}(B_d(x,\tfrac{r}{c_2}))
        \leq \frac{c_2}{c_1} C_{d,E}^m\rho_{\Tilde{d},E}(B_{\Tilde{d}}(x,r)),
    \end{equation*}
    where both inequalities in \eqref{eq7} along with the doubling condition and the monotonous behaviour of $\rho_{d,E}(x,\cdot)$ where used to get the previous estimates. 
\end{proof}

We are now in position to introduce the concept of weak porosity in spaces of homogeneous type. 

\begin{definition}
    \label{def3}
    We say that a non-empty set $E\subset X$ is \textbf{$(\sigma,\gamma)$-weakly porous with respect to $d$ and $\mu$} (or simply weakly porous, when the context is clear) if there are two constants $\sigma,\gamma\in(0,1)$ such that for every ball $B=B_d(x,r)$ there exists a finite collection of balls $\{B_d(x_i,r_i)\}_{i=1}^N$ (where $N$ depends on $B$) satisfying
    \begin{itemize}
        
        \item[\textit{(i)}] $B_d(x_i,r_i)\cap B_d(x_j,r_j)=\emptyset$ for $i\neq j$ and $B(x_i,r_i)\subset B\setminus E$ for all $1\leq i\leq N$;
        \item[\textit{(ii)}] $r_i\geq\gamma\rho_{d,E}(B)$ for every $i=1,\hdots,N$;
        
        \item[\textit{(iii)}] $r_i\leq2K_dr$ for every $i=1,\hdots,N$;
        \item[\textit{(iv)}] $\sum_{i=1}^N\mu(B_d(x_i,r_i))\geq \sigma\mu(B)$.
    \end{itemize}
\end{definition}

Let us notice that condition \textit{(iii)} does not entail a real restriction. In fact, since from \textit{(i)} $B_d(x_i,r_i)\subset B_d(x,r)$ for every $i$, we have that $B_d(x_i,r_i)=B_d(x_i,\min\{r_i,2K_dr\})$, as was already noticed by Mudarra in \cite{MUDARRA}. Furthermore, Definition~\ref{def3} reduces to the definition of weakly porous sets given in \cite{MUDARRA} for measure metric spaces when $K_d=1$. It is rather straightforward to check that the weak porosity condition on a set $E$ is also invariant by changes in equivalent quasi-distances if the maximal $E$-free hole function verifies the doubling condition for this set.

\begin{lemma}
    \label{lemmaa0}
    Suppose $E\subset X$ and let $\Tilde{d}\in\partial(X,d)$ be another quasi-distance with constants $0<c_1\leq c_2<\infty$ as in \eqref{eq5}. Then, $E$ is a $(\sigma,\gamma)$-weakly porous set with respect to $d$ and $\rho_{d,E}$ is doubling if and only if $E$ is $(\sigma_0,\gamma_0)$-weakly porous with respect to $\Tilde{d}$ and $\rho_{\Tilde{d},E}$ is doubling. 
\end{lemma}

\begin{proof}
    Once again, it is enough to check one implication, so let us assume that $E$ is $(\sigma,\gamma)$-weakly porous with respect to the quasi-distance $d$ and that $\rho_{d,E}$ is doubling. We already know, because of Corollary~\ref{coro3.5}, that $\rho_{\Tilde{d},E}$ is doubling, so all that remains is to look for the constants $\sigma_0$ and $\gamma_0$. Let $B_{\Tilde{d}}(x,r)$ be an arbitrary $\Tilde{d}$-ball in $X$. By equivalence, we have the inclusions $B_d\big(x,\tfrac{r}{c_2}\big)\subset B_{\Tilde{d}}(x,r)\subset B_d\big(x,\tfrac{r}{c_1}\big).$ Now, by hypothesis, there exists a finite collection of $d$-balls $\{B_d(z_i,r_i)\}_{i=1}^N$, each contained in $B_d(x,\frac{r}{c_2})$, such that $B_d(z_i,r_i)\cap B_d(z_j,r_j)=\emptyset$ and $B_d(z_i,r_i)\subset B_d(x,\frac{r}{c_2})\setminus E$; $r_i\geq\gamma\rho_{d,E}(B_d(x,\frac{r}{c_2}))$;  $r_i\leq 2K_d\frac{r}{c_2}$ and $\sum_{i=1}^N\mu(B_d(z_i,r_i))\geq \sigma\mu(B_d(x,\frac{r}{c_2}))$. If we denote $\Tilde{r}_i:=c_1r_i$, then, for each $1\leq i\leq N$, we have $B_{\Tilde{d}}(z_i,\Tilde{r}_i)\subset B_d(z_i,r_i)$. Let us verify that the collection of $\Tilde{d}$-balls $\{B_{\Tilde{d}}(z_i,\Tilde{r}_i)\}_{i=1}^N$ satisfies the four conditions of Definition~\ref{def3} for the $\Tilde{d}$-ball $B_{\Tilde{d}}(x,r)$.
    \begin{itemize}
        \item[\textit{(i)}] $B_{\Tilde{d}}(z_i,\Tilde{r}_i)\cap B_{\Tilde{d}}(z_j,\Tilde{r}_j)\subset B_d(z_i,r_i)\cap B_d(z_j,r_j)=\emptyset$ and $B_{\Tilde{d}}(z_i,\Tilde{r}_i)\subset B_d(z_i,r_i)\subset B_d(x,\frac{r}{c_2})\setminus E\subset B_{\Tilde{d}}(x,r)\setminus E$.
        \item[\textit{(ii)}] As $\Tilde{r}_i=c_1r_i$, if we take $m\in\mathbb{N}_0$ such that $2^{m-1}<\frac{c_2}{c_1}\leq2^m$, we have
        \begin{equation*}
            \Tilde{r}_i \geq c_1\gamma\rho_{d,E}\Big(B_d\Big(x,\frac{r}{c_2}\Big)\Big)\geq\frac{c_1\gamma}{C_{d,E}^m}\rho_{d,E}\Big(B_d\Big(x,\frac{r}{c_1}\Big)\Big)\geq\frac{c_1\gamma}{c_2C_{d,E}^m}\rho_{\Tilde{d},E}(B_{\Tilde{d}}(x,r)),
        \end{equation*}
        so $\Tilde{r}_i\geq\gamma_0\rho_{\Tilde{d},E}(B_{\Tilde{d}}(x,r))$ with $\gamma_0:=\frac{c_1\gamma}{c_2C_{d,E}^m}\in(0,1)$.
        \item[\textit{(iii)}] Recalling $\frac{K_d}{K_{\Tilde{d}}}\leq\frac{c_2}{c_1}$, we have $\Tilde{r}_i \leq 2c_1K_d\frac{r}{c_2}=2\Big(\frac{K_d}{K_{\Tilde{d}}}\Big)\frac{c_1}{c_2}K_{\Tilde{d}} r\leq 2K_{\Tilde{d}} r$.
        \item[\textit{(iv)}] Finally, with $m$ as in \textit{(ii)} and $A_d$ denoting the doubling constant for the measure of the $d$-balls,
        \begin{align*}
            \sum_{i=1}^N\mu(B_{\Tilde{d}}(z_i,\Tilde{r}_i))&\geq \frac{1}{A_{\Tilde{d}}^m}\sum_{i=1}^N\mu(B_{\Tilde{d}}(z_i,c_2r_i))\\
            &\geq \frac{1}{A_{\Tilde{d}}^m}\sum_{i=1}^N\mu(B_d(z_i,r_i))\\
            &\geq\frac{\sigma}{A_{\Tilde{d}}^m}\mu\Big(B_d\Big(x,\frac{r}{c_2}\Big)\Big) \\
            &\geq\frac{\sigma}{A_{\Tilde{d}}^mA_d^m}\mu\Big(B_d\Big(x,\frac{r}{c_1}\Big)\Big) \\
            &\geq\frac{\sigma}{(A_{\Tilde{d}}A_d)^m}\mu(B_{\Tilde{d}}(x,r)) \\
            &=:\sigma_0\mu(B_{\Tilde{d}}(x,r)),
        \end{align*}
        where, clearly, $\sigma_0\in(0,1)$.
    \end{itemize}

    Therefore, $E$ is $(\sigma_0,\gamma_0)$-weakly porous with respect to $\Tilde{d}$.
\end{proof}

\section{Proof of Theorem 1.1: $A_1$ implies weak porosity and doubling of $\rho_{d,E}$}\label{sec:proofA1impliesweakporosityAndDoubling}

In this section, we will prove that from the $A_1$ condition of $d(x,E)^{-\alpha}$ for some $\alpha>0$ we can infer the weak porosity of $E$ and the doubling of $\rho_{d,E}$. We remark that to demonstrate this first implication of Theorem~\ref{teo1} we do not need any prior results besides the basic definitions. Let us first prove the weak porosity of $E$.

\begin{theorem}
    \label{lemma1}
    Let $E$ be a non-empty subset of $X$. If there is some $\alpha>0$ such that $d(\cdot,E)^{-\alpha}\in A_1(X,d,\mu)$, then $E$ is weakly porous.
\end{theorem}

\begin{proof}
    Given some ball $B=B_d(y,r)$ in $X$, if we were to have $\rho_{d,E}(B)=0$, then arguing as in the proof of Lemma~\ref{lemma3.3}, we see that this should imply $d(x,E)=0$ for every $x\in B$. 
    However, this contradicts the integrability of $d(x,E)^{-\alpha}$ on the ball $B$. Moreover, the fact that $d(\cdot,E)^{-\alpha}\in L^1(B,d,\mu)$ implies
    \begin{equation*}
        \mu(\{x\in B:d(x,E)^{-\alpha}=\infty\})=\mu(\{x\in B:d(x,E)=0\})=\mu(B\cap\bar{E})=0.
    \end{equation*}
    \noindent Therefore, $\mu(\bar{E})=0$ and $\rho_{d,E}(B)>0$ for every ball $B$. Let us now suppose, for the sake of contradiction, that the set $E$ is not weakly porous. Picking then any pair of numbers $\sigma\in(0,1)$ and $\gamma\in(0,\frac{1}{4K_d^2})$, there should be at least one ball $B=B_d(y,r)$ for which no finite collection $\{B_d(x_i,r_i)\}_{i=1}^N$ can simultaneously satisfy properties \textit{(i)}-\textit{(iv)} in Definition~\ref{def3}. Consider now the set $D=\{x\in B:d(x,E)\geq\gamma\rho_{d,E}(B)\text{ and }d(x,X\setminus B)\geq\gamma\rho_{d,E}(B)\}$, where we set $d(x,X\setminus B)=\infty$ in the case $X\setminus B=\emptyset$. Since $\rho_{d,E}(B)>0$ and $\gamma<1$, Definition~\ref{def2} clearly indicates $D\neq\emptyset$. Next, take $\{x_i\}_{i=1}^M$ to be a maximal collection of points in $D$ satisfying $d(x_i,x_j)\geq2\gamma\rho_{d,E}(B)$ for $i\neq j$ and set $\Tilde{r}=\gamma\rho_{d,E}(B)$. Note that $\bigcup_{i=1}^M B_d(x_i,2\Tilde{r})$ contains $D$ because of the maximality property. We now extract a subfamily out of $\{B_d(x_i,\Tilde{r})\}_{i=1}^M$ as follows. Set $i_1:=1$. Assuming we have defined $i_1,\hdots,i_k$ and that $\Omega_k=\{1\leq i\leq M:B_d(x_i,\Tilde{r})\cap B_d(x_{i_j},\Tilde{r})=\emptyset,\text{ for all } 1\leq j\leq k\}$ is non-empty, pick $i_{k+1}$ as any element in this set. If instead $\Omega_k$ is empty, we take $\Tilde{M}=k$. By construction, $\{B_d(x_{i_j},r_{i_j})\}_{j=1}^{\Tilde{M}}$ is clearly a pairwise disjoint family. Furthermore, given some $z\in B_d(x_i,2\Tilde{r})$ ($1\leq i\leq M$), there exists some $1\leq j\leq \Tilde{M}$ such that $B_d(x_i,\Tilde{r})\cap B_d(x_{i_j},\Tilde{r})\neq\emptyset$. Take a point $w$ in this intersection. Then
    \begin{align*}
        d(z,x_{i_j}) &\leq K_d [ d(z,w) + d(w,x_{i_j}) ] \\
        &\leq K_d [ K_d[d(z,x_i) + d(x_i,w)] + d(w,x_{i_j}) ] \\
        &< K_d(3K_d+1)\Tilde{r}.
    \end{align*}
    We then have that $\bigcup_{j=1}^{\Tilde{M}} B_d(x_{i_j},K_d(3K_d+1)\Tilde{r})$ contains $D$ and the collection $\{B_d(x_{i_j},\Tilde{r})\}_{i=1}^{\Tilde{M}}$ is pairwise disjoint. On the other hand, since each $x_{i_j}\in D$, we also have $B_d(x_{i_j},\Tilde{r})\subset B\setminus E$. Furthermore, inequalities $\gamma\rho_{d,E}(B)= \Tilde{r}\leq 2K_d\gamma r<2K_dr$ imply that the collection $\{B(x_{i_j},\Tilde{r})\}_{j=1}^{\Tilde{M}}$ do satisfy properties \textit{(i)} through \textit{(iii)} in Definition~\ref{def3}. In order to avoid an early contradiction, property \textit{(iv)} must be false. Hence, we should have that $\sum_{j=1}^{\Tilde{M}}\mu(B_d(x_{i_j},\Tilde{r}))<\sigma \mu(B)$, and then
    \begin{align}
        \mu(D) &\leq \mu\Big(\bigcup_{j=1}^{\Tilde{M}} B_d(x_{i_j},K_d(3K_d+1)\Tilde{r})\Big) \nonumber \\
        &\leq \sum_{j=1}^{\Tilde{M}} \mu(B_d(x_{i_j},K_d(3K_d+1)\Tilde{r})) \leq A_d^m\sum_{j=1}^{\Tilde{M}} \mu(B_d(x_{i_j},\Tilde{r}))
        \label{eq8}
        <\sigma A_d^m\mu(B),
    \end{align}

    \noindent where $m\geq2$ is such that $2^{m-1}<K_d(3K_d+1)\leq2^m$. Let us now proceed to estimate the mean value of $d(\cdot,E)^{-\alpha}$ on $B$,
    \begin{align*}
        \frac{1}{\mu(B)}\int_B d(x,E)^{-\alpha}d\mu(x) &\geq \frac{1}{\mu(B)}\int_{B_d(y,\frac{r}{2K_d})}d(x,E)^{-\alpha}d\mu(x) \\
        &\geq \frac{1}{\mu(B)}\int_{B_d(y,\frac{r}{2K_d})\setminus D}d(x,E)^{-\alpha}d\mu(x).
    \end{align*}
    We claim that all points contained in the set $B_d(y,2^{-1}K_d^{-1}r)\setminus D$ are at a distance less than $\gamma\rho_{d,E}(B)$ to $E$. Write
    $B_d\big(y,\frac{r}{2K_d}\big)\setminus D = A_1\cup A_2$, where $A_1=B_d\big(y,\frac{r}{2K_d}\big)\cap\{d(x,E)<\gamma\rho_{d,E}(B)\}$ and $A_2=B_d\big(y,\frac{r}{2K_d}\big)\cap\{d(x,X\setminus B)<\gamma\rho_{d,E}(B)\}$. To prove the claim, we only need to verify that $A_2=\emptyset$. If $X\setminus B=\emptyset$, this is immediate. In the case $X\setminus B\neq\emptyset$, we start by taking $w\in B_d(y,2^{-1}K_d^{-1}r)$ and $\xi\in X\setminus B$. We have
    $$K_dd(w,\xi)\geq d(y,\xi)-K_dd(y,w)\geq r-K_dd(y,w),$$
    \noindent consequently
    $$d(w,\xi)\geq \frac{1}{K_d}r-d(y,w)\geq\frac{1}{K_d}r-\frac{1}{2K_d}r=\frac{r}{2K_d}$$
    \noindent and $d(w,X\setminus B)\geq2^{-1}K_d^{-1}r$. Then,
    \begin{align*}
        A_2 \subset \Big\{d(x,X\setminus B)\geq\frac{r}{2K_d}\Big\}\cap\{d(x,X\setminus B)<\gamma\rho_{d,E}(B)\} = \emptyset
    \end{align*}
    since $\gamma<\frac{1}{4K_d^2}$ implies $\gamma\rho_{d,E}(B)\leq2\gamma K_dr<\frac{r}{2K_d}$. Taking $k\geq1$ such that $2^{k-1}<2K_d\leq 2^k$ and using that $B_d(y,2^{-1}K_d^{-1}r)\setminus D=A_1$ and $\mu(D)\leq\sigma A^m_d\mu(B)$, we have  
    \begin{align}
        \frac{1}{\mu(B)}\int_B d(x,E)^{-\alpha}d\mu(x) &\geq\frac{1}{\mu(B)}\int_{B_d(y,\frac{r}{2K_d})\setminus D}d(x,E)^{-\alpha}dx \nonumber\\ &\geq \frac{\mu(B_d(y,2^{-1}K_d^{-1}r)\setminus D)}{\mu(B)}\gamma^{-\alpha}\rho_{d,E}(B)^{-\alpha} \nonumber \\
        &\geq \Big(\frac{\mu(B_d(y,2^{-1}K_d^{-1}r))-\mu(D)}{\mu(B)}\Big)\gamma^{-\alpha}\rho_{d,E}(B)^{-\alpha} \nonumber \\
        &\geq \Big(\frac{A_d^{-k}\mu(B)-\sigma A_d^m\mu(B)}{\mu(B)}\Big)\gamma^{-\alpha}\rho_{d,E}(B)^{-\alpha} \nonumber \\
        \label{eq9}
        &=(A_d^{-k}-A_d^m\sigma)\gamma^{-\alpha}\rho_{d,E}(B)^{-\alpha}.
    \end{align}
    Recalling once again that $\rho_{d,E}(B)$ is strictly positive, it is possible to pick $z\in B$ and $t>\frac{1}{2}\rho_{d,E}(B)$ such that $B_d(z,t)\subset B\setminus E$ and consequently
    $$\mu\Big(\Big\{x\in B:d(x,E)\geq\frac{1}{4K_d}\rho_{d,E}(B)\Big\}\Big)\geq\mu\Big(B\Big(z,\frac{1}{2K_d}t\Big)\Big)>0.$$
    In particular, this means that $\text{ess inf}_B\:d(\cdot,E)^{-\alpha}\leq 4^{\alpha}K_d^\alpha\rho_{d,E}(B)^{-\alpha}$, which combined with the hypothesis $d(\cdot,E)^{-\alpha}\in A_1(X,d,\mu)$ reads
    \begin{align}
        \frac{1}{\mu(B)}\int_B d(x,E)^{-\alpha}d\mu(x) &\leq [d(\cdot,E)^{-\alpha}]_{A_1(X,d,\mu)}\: \text{ess inf}_{x\in B}\:d(x,E)^{-\alpha} \nonumber \\
        &\leq 4^\alpha K_d^\alpha[d(\cdot,E)^{-\alpha}]_{A_1(X,d,\mu)}\:\rho_{d,E}(B)^{-\alpha}\nonumber\\
        \label{eq10}
        &=:C(\alpha)\:\rho_{d,E}(B)^{-\alpha}.
    \end{align}
    Plugging \eqref{eq9} and \eqref{eq10} together, we find that the quantity $\rho_{d,E}(B)^{-\alpha}$ cancels out and all it remains is
    \begin{equation} \label{eq11}
        \frac{(A_d^{-k}-A_d^m\sigma)}{\gamma^\alpha}\leq C(\alpha).
    \end{equation}
	Relation \eqref{eq11} must remain true for every $\sigma\in(0,1)$ and every $\gamma\in(0,(4K_d^2)^{-1})$. However, taking $\sigma<A_d^{-k-m}$ and considering the limit as $\gamma\to0^+$ makes this inequality fail. Thus, $E$ is weakly porous.
\end{proof}

\begin{theorem}
    \label{lemma2}
   Let $E$ be a non-empty subset of $X$. If there is some $\alpha>0$ such that $d(\cdot,E)^{-\alpha}\in A_1(X,d,\mu)$, then $\rho_{d,E}$ is doubling.
\end{theorem}

\begin{proof}
    Let $B=B_d(y,r)$ denote some arbitrary ball in $X$. Also, write $\lambda B$ to refer to the ball $B_d(y,\lambda r)$, for a given $\lambda>0$. We will consider two cases separately: whether $\lambda_0B\cap E=\emptyset$ or $\lambda_0B\cap E\neq\emptyset$, where $\lambda_0^{-1}:=K_d[2K_d+1]$. In the first case, we have $B_d(y,\lambda_0r)\subset B\setminus E$, from what we see $\lambda_0r\in\Lambda(y,r;d,E)$, and so $\rho_{d,E}(B)\geq\lambda_0r$. Since $\rho_{d,E}(2B)\leq 4K_dr$ by definition of $\rho_{d,E}$, the bound $\rho_{d,E}(2B)\leq 4K_d\lambda_0^{-1} \rho_{d,E}(B)=4K_d^2[2K_d+1]\rho_{d,E}(B)$ becomes clear. Assuming now $\lambda_0B\cap E\neq\emptyset$ we note, by using \eqref{eq10}, that
    \begin{align}
        \frac{1}{\mu(B)}\int_B d(x,E)^{-\alpha} d\mu(x) &\leq
        \frac{A_d}{\mu(2B)}\int_{2B} d(x,E)^{-\alpha} d\mu(x) \nonumber \\
        &\leq A_d[d(\cdot,E)^{-\alpha}]_{A_1(X,d,\mu)}\:\text{ess inf}_{x\in 2B}\:d(x,E)^{-\alpha} \nonumber \\
        &\leq A_d[d(\cdot,E)^{-\alpha}]_{A_1(X,d,\mu)}C\rho_{d,E}(2B)^{-\alpha} \nonumber \\
        \label{eq12}
        &=: \Tilde{C}\rho_{d,E}(2B)^{-\alpha}.
    \end{align}

    \noindent On the other hand,
    \begin{align}
        \frac{1}{\mu(B)}\int_B d(x,E)^{-\alpha} d\mu(x) &\geq \frac{1}{\mu(B)}\int_{\lambda_0B} d(x,E)^{-\alpha} d\mu(x) \nonumber \\
        &\geq \label{eq13} \frac{\mu(\lambda_0B)}{\mu(B)}\:\text{ess inf}_{x\in\lambda_0B}\:d(x,E)^{-\alpha}.
    \end{align}
    If $x\in\lambda_0B\setminus\Bar{E}$ and we denote $t=d(x,E)$, then $0<t\leq 2K_d\lambda_0r$ (recall $\lambda_0B\cap E\neq\emptyset$). Moreover, if $X\setminus B\neq\emptyset$,
    \begin{align*}
        d(x,X\setminus B)\geq \frac{1}{K_d}d(y,X\setminus B)-d(y,x)\geq\Big(\frac{1}{K_d}-\lambda_0\Big)r &=\Big(\frac{1}{K_d}-\frac{1}{K_d[2K_d+1]}\Big)r \\
        &=\frac{2K_d+1-1}{K_d[2K_d+1]}r \\
        &=2K_d\lambda_0r \\
        &\geq t.
    \end{align*}
    So $B_d(x,t)\subset B\setminus E$. Note that this inclusion holds trivially if $X\setminus B=\emptyset$. Since $t<2K_dr$, we have $t\in\Lambda(y,r;d,E)$, this gives $\rho_{d,E}(B)\geq t = d(x,E)$, so that $\rho_{d,E}(B)^{-\alpha}\leq d(x,E)^{-\alpha}$ for every $x\in\lambda_0B\setminus\Bar{E}$. Hence, since $\mu(\Bar{E})=0$, we get
    \begin{equation}
        \label{eq13b}
        \rho_{d,E}(B)^{-\alpha}\leq \text{inf ess}_{x\in\lambda_0B}\:d(x,E)^{-\alpha}.
    \end{equation}
    Now, \eqref{eq12}, \eqref{eq13} and \eqref{eq13b} give us the desired inequality
    \begin{align*}
        \frac{\mu(\lambda_0B)}{\mu(B)}\rho_{d,E}(B)^{-\alpha} \leq \Tilde{C} \rho_{d,E}(2B)^{-\alpha},
    \end{align*}
    \noindent or simply
    \begin{equation*}
        \rho_{d,E}(2B) \leq \Tilde{C} \rho_{d,E}(B)
    \end{equation*}
    by renaming $\Tilde{C}$.
\end{proof}

\section{Proof of Theorem 1.1: weak porosity of $E$ and doubling of $\rho_{d,E}$ implies the $A_1$ condition of $d(\cdot,E)^{-\alpha}$ for some $\alpha>0$}\label{sec:proofWeakPorosityAndDoublingimpliesA1}

Let us start this section by briefly explaining the heuristic idea of our approach. First, given $(X,d,\mu)$ a space of homogeneous type we change the quasi-distance $d$ to $\delta$ of Definition~\ref{def4}. Now $\delta$-balls are quite smooth in the sense that they are subspaces of homogeneous type uniformly. In particular, we are able to use Whitney covering lemmas restricted to $\delta$-balls with uniform constants. So in the search for the estimate
$$\fint_{B_\delta(x_0,r_0)} \delta(x,E)^{-\alpha}d\mu(x)\leq C\:\text{ess inf}_{B_\delta(x_0,r_0)}\delta(\cdot,E)^{-\alpha},$$
uniformly in $x_0\in X$ and $r_0>0$, we shall show that the measure of $\varepsilon$-neighbourhoods of $E$ in $B_\delta(x_0,r_0)$ decay exponentially for $\varepsilon$ approaching zero. This fact will, in turn, allow us to obtain a bound for the mean value of $d(x,E)^{-\alpha}$ for some positive $\alpha$. At this point, it is worthy to mention that this exponential behavior can also be seen as a positive lower bound for the so-called Muckenhoupt exponent introduced by the authors of \cite{ANDERSON,MUDARRA}.

In order to prove that (I) implies (II) in Theorem~\ref{teo1}, let us start by stating the lemmas that are the key points of our proof. Given a space of homogeneous type $(X,d,\mu)$, we shall use $\delta$ to denote the quasi-metric introduced in Definition~\ref{def4} associated with $d$. With $K_\delta$ we shall denote the triangular constant of $(X,\delta)$, and with $A_\delta$ the doubling constant of $(X,\delta,\mu)$.

\begin{lemma}
    \label{lemma5.1}
    Let $E$ be any non-empty subset of $X$ such that $\rho_{\delta,E}$ is positive and doubling. Let $B_\delta^0:=B_\delta(x_0,r_0)$ be some ball in $X$ with $B_\delta^0\cap E\neq\emptyset$. For $0<\varepsilon<\rho_{\delta,E}(B_\delta^0)$ set $E_\delta(\varepsilon)=\{x\in X:\delta(x,E)<\varepsilon\}$ and $F_\delta(\varepsilon)=E_\delta(\varepsilon)\cap B_\delta^0$. Then, there exist a countable index set $I$, for each $i\in I$ a point $z_i\in X$ and a positive number $t_i$ such that
    \begin{itemize}
        \item[\textit{(1)}] balls in $\{B_\delta(z_i,t_i)\}_{i\in I}$ are pairwise disjoint and $\bigcup_{i\in I}B_\delta(z_i,t_i)\subset F_\delta(\varepsilon)$;

        \item[\textit{(2)}] $F_\delta(\varepsilon)\subset\bigcup_{i\in I}B_\delta\Big(z_i,\frac{5K_\delta^2}{\beta}t_i\Big),$ where $\beta$ is the constant provided in (iii) of Theorem~\ref{teo2};
        
        \item[\textit{(3)}] for each $i\in I$ there exists $y_i\in B_\delta^0\setminus F_\delta(\varepsilon)$ such that $\delta(z_i,y_i)\leq\frac{13K_\delta^3}{\beta}t_i$;
        \item[\textit{(4)}] for $0<\eta\leq\frac{\beta}{12K_\delta^3}$, the set $I_\eta=\{i\in I:t_i\geq\eta\varepsilon\}$ is non-empty and for $i\in I_\eta$,
        
        \begin{equation}
        \label{eq15} \rho_{\delta,E}(B_\delta(z_i,t_i)) \geq \Theta(\eta)\varepsilon,\:\forall i\in I_\eta,
    \end{equation}
    with $\Theta(\eta):=\Big[\frac{2K_\delta}{\eta}+\frac{26K_\delta^4}{\beta}\Big]^{-\log_2C_{\delta,E}}$ and where $C_{\delta,E}$ is, as in Section~\ref{sec:weaklyPorousSets}, the doubling constant for $\rho_{\delta,E}$.
        \end{itemize}
\end{lemma}

\begin{lemma}
    \label{lemma5.2}
    Let $E$ be a $(\sigma_0,\gamma_0)$-weakly porous set in $X$ with respect to $\delta$ such that $\rho_{\delta,E}$ is doubling. Given $B_\delta^0=B_\delta(x_0,r_0)$ such that $B_\delta^0\cap E\neq\emptyset$ and $0<\varepsilon<\rho_{\delta,E}(B_\delta^0)$, there exist constants $0<p,q<1$, independent of $\varepsilon$, such that
    \begin{itemize}
        \item[\textit{(5)}] $\mu(F_\delta(p\varepsilon))\leq q\mu(F_\delta(\varepsilon))$;
        \item[\textit{(6)}] $\mu(F_\delta(\frac{1}{2}p^k\rho_{\delta,E}(B^0_\delta)))\leq q^k\mu(B_\delta^0)$, for every nonnnegative integer $k$;

        \item[\textit{(7)}] $\mu(\Bar{E})=0$.
    \end{itemize}
\end{lemma}

We shall give the proofs of Lemmas~\ref{lemma5.1} and \ref{lemma5.2} after proving the main result contained in the next statement.

\begin{theorem}
    \label{teo5.3}
    Let $(X,d,\mu)$ be a space of homogeneous type such that the $d$-balls are open sets.
    Let $E$ be a $(\sigma,\gamma)$-weakly porous set in $X$ with respect to $d$ such that $\rho_{d,E}$ is doubling. Then, there exists $\alpha>0$ such that $w(x)=d(\cdot,E)^{-\alpha}$ belongs to $A_1(X,d,\mu)$. The constants $\alpha$ and $[w]_{A_1(X,d,\mu)}$ depend on $\sigma$, $\gamma$, $K_d$, $A_d$, and $C_{d,E}$.
\end{theorem}

\begin{proof}
    For the given quasi-distance $d$ on $X$, applying Theorem~\ref{teo2}, we obtain that the new quasi-distance $\delta$ given by Definition~\ref{def4} belongs to $\partial(X,d)$. Hence, since $(X,d,\mu)$ is $(\sigma,\gamma)$-weakly porous and the function $\rho_{d,E}$ satisfies the doubling condition with constant $C_{d,E}$, from Corollary~\ref{coro3.5} and Lemma~\ref{lemmaa0}, we have that $(X,\delta,\mu)$ is $(\sigma_0,\gamma_0)$-weakly porous and $\rho_{\delta,E}$ satisfies the doubling condition with constant $C_{\delta,E}$. The parameters $\sigma_0$, $\gamma_0$, and $C_{\delta,E}$ depend on $K_d$, $A_d$, $\sigma$, $\gamma$, and $C_{d,E}$. On the other hand, from Proposition~\ref{propo2.4} it is enough to prove that for some $\alpha>0$, $\delta(\cdot,E)^{-\alpha}\in A_1(X,\delta,\mu)$ so that we are looking for positive constants $C$ and $\alpha$ such that for every ball $B_\delta^0=B_\delta(x_0,r_0)$, the inequality
    \begin{equation}
        \label{eq5.3.1}
        \fint_{B_\delta(x_0,r_0)} \delta(x,E)^{-\alpha}d\mu(x)\leq C\:\text{ess inf}_{B_\delta(x_0,r_0)}\delta(\cdot,E)^{-\alpha}
    \end{equation}
    holds. In order to prove this inequality we consider three cases regarding the relation of $B_\delta^0$ and $E$. More precisely, consider the three cases.
    \begin{itemize}
    \item[(A)] $B_\delta(x_0,r_0)\cap E=\emptyset\land \delta(B_\delta(x_0,r_0),E)\geq4K_\delta^2r_0$;
    \item[(B)] $B_\delta(x_0,r_0)\cap E=\emptyset\land \delta(B_\delta(x_0,r_0),E)<4K_\delta^2r_0$;
    \item[(C)] $B_\delta(x_0,r_0)\cap E\neq\emptyset$,
    \end{itemize}
    where $\delta(A,B):=\inf_{a\in A,b\in B}\delta(a,b)$ denotes the distance with respect to $\delta$ of the sets $A$ and $B$. Let us start with the case (A). Take $x,y\in B_\delta^0$ and $\alpha>0$ to be chosen later. Then, $\delta(x,E)\geq \frac{1}{K_\delta}\delta(y,E)-\delta(x,y)\geq \frac{1}{K_\delta}\delta(y,E)-2K_\delta r_0\geq \frac{1}{K_\delta}\delta(y,E)-\frac{1}{2K_\delta}\delta(B_\delta^0,E)\geq\frac{1}{2K_\delta}\delta(y,E)$. So that $\delta(x,E)^{-\alpha}\leq(2K_\delta)^\alpha\:\text{inf ess}_{y\in B_\delta^0}\delta(y,E)^{-\alpha}$ for every $\alpha>0$ and every $x\in X$, thus
    \begin{equation}
    \label{eq17}
    \fint_{B_\delta^0}\delta(x,E)^{-\alpha}d\mu(x)\leq(2K_\delta)^\alpha\:\text{inf ess}_{y\in B_\delta^0}\delta(y,E)^{-\alpha}.
    \end{equation}

    \noindent Now, (B) can be obtained from (C) noting that, with $m$ such that $2^{m-1}\leq5K_\delta^3\leq2^m$, we have
    \begin{align}
    \fint_{B_\delta^0}\delta(x,E)^{-\alpha}d\mu(x) &\leq \frac{\mu(5K_\delta^3B_\delta^0)}{\mu(B_\delta^0)}\fint_{5K_\delta^3B_\delta^0}\delta(x,E)^{-\alpha}d\mu(x) \nonumber \\
    &\leq A_\delta^m\fint_{5K_\delta^3B_\delta^0}\delta(x,E)^{-\alpha}d\mu(x). \nonumber 
    \end{align}
    Observe now that $5K_\delta^3B_\delta^0\cap E\neq\emptyset$. Then, assuming that \eqref{eq5.3.1} holds in case (C) for some appropriate values of $\alpha>0$, we have
    \begin{equation}
        \label{eqq2}
        \fint_{B_\delta^0}d(x,E)^{-\alpha}d\mu(x)\leq CA^m_\delta\:\text{ess inf}_{5K_\delta^3B_\delta^0}\delta(\cdot,E)^{-\alpha}\leq CA^m_\delta\:\text{ess inf}_{B_\delta^0}\delta(\cdot,E)^{-\alpha},
    \end{equation}
    as desired. Let us finally consider the case (C). Assume that $B_\delta^0\cap E\neq\emptyset$ and write $\rho=\rho_{\delta,E}(B_\delta^0)$. Let us split the integral in question using the notation $E_\delta(\cdot)$ and $F_\delta(\cdot)$ introduced in Lemma~\ref{lemma5.2}.
    \begin{align}
    \int_{B_\delta^0}\delta(x,E)^{-\alpha}d\mu(x)&= \int_{B_\delta^0\setminus E_\delta(\frac{1}{2}\rho)}\delta(x,E)^{-\alpha}d\mu(x) + \int_{F_\delta(\frac{1}{2}\rho)}\delta(x,E)^{-\alpha}d\mu(x) \nonumber \\
    &\leq 
    \label{eqq0}
    \Big(\frac{\rho}{2}\Big)^{-\alpha}\mu\Big(B_\delta^0\setminus E_\delta\Big(\frac{1}{2}\rho\Big)\Big) + \int_{F_\delta(\frac{1}{2}\rho)}\delta(x,E)^{-\alpha}d\mu(x).
    \end{align}
    To find a proper bound for the second term on the right-hand side notice that the set where the integral is taken can be expressed as
    \begin{align*}
     F_\delta(2^{-1}\rho)&=\{x\in B_\delta^0:\delta(x,E)<2^{-1}\rho\} \\   
     &=\{x\in B_\delta^0:\delta(x,E)=0\}\cup \bigcup_{k=0}^\infty\{x\in B_\delta^0:2^{-1}p^{k+1}\rho\leq \delta(x,E)<2^{-1}p^k\rho\} \\
     &=(\Bar{E}\cap B_\delta^0)\cup\bigcup_{k=0}^\infty F_\delta(2^{-1}p^k\rho)\setminus F_\delta(2^{-1}p^{k+1}\rho),
    \end{align*}
    where $0<p<1$ is provided by Lemma~\ref{lemma5.2}. Furthermore, from \textit{(7)} in Lemma~\ref{lemma5.2}, $\mu(\Bar{E})=0$. Hence, from \textit{(6)} in the same lemma,
    \begin{align*}
    \int_{F_\delta(\frac{1}{2}\rho)}\delta(x,E)^{-\alpha}d\mu(x)
    &=\sum_{k=0}^\infty \int_{F_\delta(p^k\frac{\rho}{2})\setminus F_\delta(p^{k+1}\frac{\rho}{2})}\delta(x,E)^{-\alpha}d\mu(x) \\
    &\leq \sum_{k=0}^\infty \Big(p^{k+1}\frac{\rho}{2}\Big)^{-\alpha} \mu\Big( F_\delta\Big(p^k\frac{\rho}{2}\Big)\setminus F_\delta\Big(p^{k+1}\frac{\rho}{2}\Big) \Big) \\
    &\leq \Big(p\frac{\rho}{2}\Big)^{-\alpha} \sum_{k=0}^\infty (p^{-\alpha})^k \mu\Big( F_\delta\Big(p^k\frac{\rho}{2}\Big)\Big) \\
    &\leq \Big(p\frac{\rho}{2}\Big)^{-\alpha} \mu(B_\delta^0) \sum_{k=0}^\infty \Big(\frac{q}{p^\alpha}\Big)^k \\
    &\leq C(\alpha)\mu(B_\delta^0)\rho^{-\alpha},
    \end{align*}
    if we take $\alpha>0$ sufficiently small as to make $p^{-\alpha}q<1$. Now, since $B_\delta^0\cap E\neq\emptyset$ and $\mu(\Bar{E})=0$, Lemma~\ref{lemma3.3} implies
    \begin{equation}
        \label{eqq1}
        \int_{F_\delta(\frac{1}{2}\rho)}\delta(x,E)^{-\alpha}d\mu(x)\leq \Tilde{C}(\alpha)\mu(B_\delta^0)\text{inf ess}_{B^0_\delta}\delta(\cdot,E)^{-\alpha}.
    \end{equation}
    Combining then \eqref{eq17}, \eqref{eqq2}, \eqref{eqq0} and \eqref{eqq1} and taking the maximum over all constants found, it follows that $d(\cdot,E)^{-\alpha}$ belongs to the class $A_1(X,\delta,\mu)$ for every $\alpha>0$ such that $p^{-\alpha}q<1$.
\end{proof}

In the following, we give the proofs of Lemmas~\ref{lemma5.1} and \ref{lemma5.2} which made it possible to obtain the main result of this section.

\begin{proof}[Proof of Lemma~\ref{lemma5.1}]
    Consider the set $\Omega_\varepsilon:=F_\delta(\varepsilon)$ seen as a subset of the space $(B_\delta^0,\delta^{*},\mu^{*})$, where $\delta^*$ and $\mu^*$ are the restrictions to $B^0_\delta$ of $\delta$ and $\mu$, respectively. This set is open, bounded, and proper in $B^0_\delta$. In fact, since $0<\varepsilon<\rho_{\delta,E}(B^0_\delta)$ we may take $\zeta$ such that $\varepsilon<\zeta<\rho_{\delta,E}(B^0_\delta)$ and a $\delta$-ball $B_\delta(x,\zeta)$ which does not intersect with $E$ and is contained in $B^0_\delta$. Since, $\varepsilon<\zeta$, it is clear that the center $x$ of this ball does not belong to $E_\delta(\varepsilon)$, hence $x\in B^0_\delta\setminus\Omega_\varepsilon=B^0_\delta\setminus F_\delta(\varepsilon)$. Apply now Lemma~\ref{lemma0}, with triangular constant $K=K_\delta$, to $\Omega_\varepsilon$ to obtain a family of $\delta^*$-balls $\{B_{\delta^*}(x_i,r_i)\}_{i\in I}$ satisfying (a)-(d) of Lemma~\ref{lemma0}. Notice that $B_{\delta^*}(x,t)=B_\delta(x,t)\cap B_\delta^0$, where $B_\delta(x,t)$ is the $\delta$-ball in $X$ centered at $x\in B^0_\delta$ with radius $t$. We then have
    \begin{equation}
        \label{eq27}
        \bigcup_{i\in I}B_{\delta^*}(x_i,r_i)\subset F_\delta(\varepsilon) = \bigcup_{i\in I} B_{\delta^*}(x_i,4K_\delta r_i). 
    \end{equation}
    Fixed some index $i\in I$, as $B_\delta^0\setminus\Omega_\varepsilon\neq\emptyset$ and $B_{\delta^*}(x_i,r_i)\subset\Omega_\varepsilon$, we have $r_i\leq 2K_\delta r_0$, so by \textit{(iii)} in Theorem~\ref{teo2} we can find some $z_i\in X$ such that $B_\delta(z_i,\beta r_i)\subset B_{\delta^*}(x_i,r_i)$. As $z_i\in B_{\delta^*}(x_i,r_i)$, we also have $B_{\delta^*}(x_i,4K_\delta r_i)\subset B_\delta(z_i,5K_\delta^2r_i)$.
    Then,
    \begin{equation}
        \label{eq22}
        \bigcup_{i\in I}B_\delta(z_i,t_i)\subset F_\delta(\varepsilon)\subset\bigcup_{i\in I}B_\delta\Big(z_i,\frac{5K_\delta^2}{\beta}t_i)\Big),
    \end{equation}
    where $t_i:=\beta r_i$ and the $B_\delta(z_i,t_i)$ are pairwise disjoint because, from \textit{(a)} in Lemma~\ref{lemma0}, the balls $B_{\delta^*}(x_i,r_i)$ are pairwise disjoint. The above inclusion proves \textit{(1)} and \textit{(2)} in Lemma~\ref{lemma5.1}. In order to prove \textit{(3)}, notice that item \textit{(d)} of Lemma~\ref{lemma0} implies that for each $i\in I$ there exists some $y_i\in B^0_\delta\setminus F_\delta(\varepsilon)$ such that $\delta^*(x_i,y_i)\leq12K_\delta^2r_i$. From this we see $\delta(z_i,y_i)\leq K_\delta[\delta(z_i,x_i)+\delta(x_i,y_i)]\leq K_\delta[r_i+12K_\delta^2r_i]\leq 13K_\delta^3\beta^{-1}t_i$ for every $i\in I$.
    Let us finally prove \textit{(4)}. Given some $e\in E\cap B_\delta^0$, because of \eqref{eq27} there must exist $i\in I$ such that $e\in B_{\delta^*}(x_i,4K_\delta r_i)$, therefore (using \textit{(c)} in Lemma~\ref{lemma0})
    \begin{align*}
        \varepsilon\leq \delta^*(e,B^0_\delta\setminus\Omega_\varepsilon)\leq12K^3_\delta r_i=\frac{12K_\delta^3}{\beta}t_i,
    \end{align*}
    so that $t_i\geq\frac{\beta}{12K_\delta^3}\varepsilon$ when $B_{\delta^*}(x_i,4K_\delta r_i)\cap E\neq\emptyset$. In particular, if $0<\eta\leq(12K_\delta^3)^{-1}\beta$ and we define the index set $I_\eta:=\{i\in I:t_i\geq\eta\varepsilon\}$, then $I_\eta\neq\emptyset$. Furthermore, if $i\in I_\eta$, as $B_\delta(y_i,\varepsilon)\cap E=\emptyset$, $B_\delta(y_i,\varepsilon)\subset  B_\delta(z_i,[K_\delta\eta^{-1}+13K_\delta^4\beta^{-1}]t_i)$ and $\varepsilon\leq\eta^{-1}t_i\leq2K_\delta[K_\delta\eta^{-1}+13K_\delta^4\beta^{-1}]t_i$, we have $\varepsilon\in\Lambda(z_i,[K_\delta\eta^{-1}+13K_\delta^4\beta^{-1}]t_i;\delta,E)$ and
    \begin{align*}
        \varepsilon 
        &\leq \rho_{\delta,E}\Big(B_\delta(z_i,[K_\delta\eta^{-1}+13K_\delta^4\beta^{-1}]t_i)\Big) \\
        &\leq C_{\delta,E}^m\rho_{\delta,E}(B_\delta(z_i,t_i)),
    \end{align*}
    \noindent where $m$ is the integer satisfying $2^{m-1}<\frac{K_\delta}{\eta} + \frac{13K_\delta^4}{\beta}\leq2^m$. These inequalities prove \eqref{eq15}.
\end{proof}

\begin{proof}[Proof of Lemma~\ref{lemma5.2}]
    Let $0<\varepsilon<\rho_{\delta,E}(B^0_\delta)$. For $\eta>0$, consider the function $\Theta(\eta)=\Big[\tfrac{2K_\delta}{\eta}+\tfrac{26K_\delta^4}{\beta}\Big]^{-\log_2C_{\delta,E}}$ introduced in Lemma~\ref{lemma5.1}. Since the constant $C_{\delta,E}$ is greater than one, because $E$ is weakly porous, the exponent $-\log_2C_{\delta,E}$ is negative and $\Theta$ is an increasing function of $\eta$. Also, $\lim_{\eta\to0}\Theta(\eta)=0$ and $\lim_{\eta\to\infty}\Theta(\eta)=(\frac{\beta}{26K_\delta^4})^{\log_2C_{\delta,E}}<1$. Set
    \begin{equation}
        \label{eq24}
        \eta_0:=\sup\Bigg\{\eta\in\Big(0,\frac{\beta}{12K_\delta^3}\Big]:\eta\leq\frac{\frac{1}{K_\delta}-\frac{\gamma_0}{2K_\delta}\Theta(\eta)}{K_\delta+13K_\delta^4\beta^{-1}}\Bigg\},
    \end{equation}
    so that $\eta_0$ is well defined as the largest positive number satisfying the inequality appearing in \eqref{eq24}.
    Since $0<\eta_0\leq\frac{\beta}{12K_\delta^3}$, the balls $B_\delta(z_i,t_i)$ of Lemma~\ref{lemma5.1} with $i\in I_{\eta_0}$ satisfy
    \begin{equation}
        \label{eq25}
        \rho_{\delta,E}(B_\delta(z_i,t_i)) \geq \Theta(\eta_0)\varepsilon.
    \end{equation}
    \noindent Having this into account, take $\varepsilon':=\frac{\gamma_0}{2K_\delta}\Theta(\eta_0)\varepsilon$. Notice $\varepsilon'<\varepsilon$ as $\Theta(\eta)<1$ for every $\eta>0$. We now proceed to estimate the measure of the set $F_\delta(\varepsilon)\setminus F_\delta(\varepsilon')$ by using \textit{(1)} in Lemma~\ref{lemma5.1},
    \begin{align*}
        \mu(F_\delta(\varepsilon)\setminus F_\delta(\varepsilon'))\geq \mu\Big( 
        \bigcup_{i\in I}B_\delta(z_i,t_i)\setminus F_\delta(\varepsilon')\Big)
        =\sum_{i\in I} \mu(B_\delta(z_i,t_i)\setminus E_\delta(\varepsilon')).
    \end{align*}
    Let us analyze individually the measures of sets $B_\delta(z_i,t_i)\setminus E_\delta(\varepsilon')$ separating in cases according to whether $i\in I_{\eta_0}$ or $i\in I\setminus I_{\eta_0}$. In the case of the smaller balls, i.e. when $i\in I\setminus I_{\eta_0}$, we invoke \textit{(3)} in Lemma~\ref{lemma5.1} to obtain a point $y_i\in B_\delta^0\setminus F_\delta(\varepsilon)$ such that $\delta(z_i,y_i)\leq 13K_\delta^3\beta^{-1}t_i$. Let us observe that $K_\delta^{-1}\varepsilon-\varepsilon'=K_\delta^{-1}\varepsilon(1-\frac{\gamma_0}{2}\Theta(\eta_0))>0$, so we claim that $B_\delta(y_i,K_\delta^{-1}\varepsilon-\varepsilon')\cap E_\delta(\varepsilon')=\emptyset$. To see this, take $w\in B_\delta(y_i,K_\delta^{-1}\varepsilon-\varepsilon')$ and $e\in E$. It follows that $\delta(w,e)\geq K_\delta^{-1}\delta(y_i,e)-\delta(y_i,w)\geq K_\delta^{-1}\varepsilon-K_\delta^{-1}\varepsilon+\varepsilon'=\varepsilon'$, since $B_\delta(y_i,\varepsilon)\cap E=\emptyset$. Hence, $w\notin E_\delta(\varepsilon')$. We now claim that $B_\delta(z_i,t_i)\subset B_\delta(y_i,K_d^{-1}\varepsilon-\varepsilon')$, so in particular $B_\delta(z_i,t_i)\setminus E_\delta(\varepsilon')=B_\delta(z_i,t_i)$. To prove this take $w\in B_\delta(z_i,t_i)$ and, recalling $t_i<\eta_0\varepsilon$, we have
    \begin{align*}
        \delta(w,y_i) &\leq K_\delta[\delta(w,z_i)+\delta(z_i,y_i)] \\
        &< t_i [K_\delta + 13K_\delta^4\beta^{-1}] \\
        &< \eta_0\varepsilon [K_\delta + 13K_\delta^4\beta^{-1}] \leq \varepsilon\Big( \frac{1}{K_\delta}-\frac{\gamma_0}{2K_\delta}\Theta(\eta_0) \Big) = \frac{\varepsilon}{K_\delta}-\varepsilon',
    \end{align*}
    as claimed. This means $\mu(B_\delta(z_i,t_i)\setminus E_\delta(\varepsilon'))=\mu(B_\delta(z_i,t_i))$ for every $i\in I\setminus I_{\eta_0}$. On the other hand, if $i\in I_{\eta_0}$, being $E$ a weakly porous set in $X$ with respect to $\delta$, we can find a finite sequence of pairwise disjoint balls $\{B_\delta(z^i_j,s^i_j)\}_{j=1}^N$ such that
    \begin{itemize}
        
         \item[\textit{(i)}] $B_\delta(z^i_j,s^i_j)\subset B_\delta(z_i,t_i)\setminus E$ for all $1\leq j\leq N$;
        \item[\textit{(ii)}] $s^i_j\geq\gamma_0 \rho_{\delta,E}(B_\delta(z_i,t_i))$, for every $j=1,\hdots,N$;
        \item[\textit{(iii)}] $s^i_j\leq2K_\delta t_i$ for every $j=1,\hdots,N$;
        \item[\textit{(iv)}] $\sum_{j=1}^N\mu(B_\delta(z^i_j,s^i_j))\geq \sigma_0\mu(B_\delta(z_i,t_i))$.
    \end{itemize}
    From \textit{(ii)} above and \eqref{eq25}, we have $s^i_j\geq\gamma_0\Theta(\eta_0)\varepsilon$, which in turn implies $\frac{1}{K_\delta}s^i_j-\varepsilon'=\frac{1}{K_\delta}(s^i_j-\frac{\gamma_0}{2}\Theta(\eta_0)\varepsilon)\geq\frac{1}{2K_\delta}s^i_j$. Also, we have that $B_\delta(z^i_j,K_\delta^{-1}s^i_j-\varepsilon')\cap E_\delta(\varepsilon')=\emptyset$. To see this, we proceed as in some previous step and take $w\in B_\delta(z^i_j,K_\delta^{-1}s^i_j-\varepsilon')$ and $e\in E$. It follows that $\delta(w,e)\geq K_\delta^{-1}\delta(z^i_j,e)-\delta(z^i_j,w)\geq K_\delta^{-1}s^i_j-K_\delta^{-1}s^i_j+\varepsilon'=\varepsilon'$, thus $w\notin E_\delta(\varepsilon')$. It turns out that for $i\in I_{\eta_0}$,
    \begin{align*}
        \mu(B_\delta(z_i,t_i)\setminus E_\delta(\varepsilon')) &\geq \mu\Big( \bigcup_{j=1}^N B_\delta(z^i_j,s^i_j) \setminus E_\delta(\varepsilon') \Big) \\
        &= \sum_{j=1}^N \mu( B_\delta(z^i_j,s^i_j) \setminus E_\delta(\varepsilon')) \\
        &\geq \sum_{j=1}^N \mu\Big( B_\delta\Big(z^i_j,\frac{s^i_j}{K_\delta}-\varepsilon'\Big)\Big) \\
        &\geq \sum_{j=1}^N \mu\Big( B_\delta\Big(z^i_j,\frac{1}{2K_\delta} s^i_j\Big) \Big) \geq \frac{1}{A_\delta^k} \sum_{j=1}^N \mu\Big( B_\delta(z^i_j,s^i_j) \Big) \geq \frac{\sigma_0}{A_\delta^k} \mu(B_\delta(z_i,t_i))
    \end{align*}
    where $2^{k-1}<2K_\delta\leq2^k$.
    Putting all together, we can continue our evaluation of $\mu(F_\delta(\varepsilon)\setminus F_\delta(\varepsilon'))$ as follows,
    \begin{align*}
        \mu(F_\delta(\varepsilon)\setminus F_\delta(\varepsilon')) &\geq \sum_{i\in I\setminus I_{\eta_0}} \mu(B_\delta(z_i,t_i)\setminus E_\delta(\varepsilon')) + \sum_{i\in I_{\eta_0}} \mu(B_\delta(z_i,t_i)\setminus E_\delta(\varepsilon')) \\
        &\geq \sum_{i\in I\setminus I_{\eta_0}} \mu(B_\delta(z_i,t_i)) + \frac{\sigma_0}{A_\delta^k}\sum_{i\in I_{\eta_0}} \mu(B_\delta(z_i,t_i)) \\
        &\geq \frac{\sigma_0}{A_\delta^k}\sum_{i\in I}\mu(B_\delta(z_i,t_i)) \\
        &\geq \frac{\sigma_0}{A_\delta^{k+l}}\sum_{i\in I}\mu\Big(B_\delta\Big(z_i,\frac{5K_\delta^2}{\beta}t_i\Big)\Big)\\
        &\geq \frac{\sigma_0}{A_\delta^{k+l}} \mu(F_\delta(\varepsilon)),
    \end{align*}
    where $l$ was chosen to verify $2^{l-1}<\frac{5K_\delta^2}{\beta}\leq2^l$.
    If we take $p:=\frac{\gamma_0}{2K_\delta}\Theta(\eta_0)$ and $q:=1-\sigma_0 A_\delta^{-k-l}$, then $0<p,q<1$, $\varepsilon'=p\varepsilon$ and the inequality $\mu(F_\delta(\varepsilon)\setminus F_\delta(\varepsilon'))\geq\frac{\sigma_0}{A_\delta^{k+l}} \mu(F_\delta(\varepsilon))$ can be rewritten as
    $$\mu(F_\delta(\varepsilon))-\mu(F_\delta(p\varepsilon))\geq (1-q)\mu(F_\delta(\varepsilon)),$$
    which can be rearranged to get \textit{(5)}. Item \textit{(6)} easily follows by applying \textit{(5)} repeatedly, starting with $\varepsilon=\frac{p^k}{2}\rho_{\delta,E}(B^0_\delta)$,
    \begin{align*}
        \mu\Big(F_\delta\Big(\frac{p^k}{2}\rho_{\delta,E}(B^0_\delta)\Big)\Big)\leq q\mu\Big(F_\delta\Big(\frac{p^{k-1}}{2}\rho_{\delta,E}(B^0_\delta)\Big)\Big)\leq\ldots\leq q^k\mu\Big(F_\delta\Big(\frac{\rho_{\delta,E}(B^0_\delta)}{2}\Big)\Big)\leq q^k\mu(B^0_\delta).
    \end{align*}
    Let us finally prove \textit{(7)}. Notice first that $\Bar{E}\subset\bigcup_{e\in E}B_\delta(e,r)=E_\delta(r)$ for any choice of positive $r$, so it suffices to check that $\mu(\Bar{E}\cap B_\delta(e,r))=0$ for every $e\in E$ and $r>0$. Given such a ball $B_\delta:=B_\delta(e,r)$, since its center belongs to $E$, clearly $B_\delta\cap E\neq\emptyset$ and we can apply \textit{(6)} with $F_\delta(\varepsilon):=E_\delta(\varepsilon)\cap B_\delta$. Indeed, for every $k\in\mathbb{N}$,
    \begin{align*}
        \mu(\Bar{E}\cap B_\delta)\leq\mu\Big(E_\delta\Big(\frac{p^k}{2}\rho_{\delta,E}(B_\delta)\Big)\cap B_\delta\Big)=\mu\Big(F_\delta\Big(\frac{p^k}{2}\rho_{\delta,E}(B_\delta)\Big)\Big)\leq q^k\mu(B_\delta),
    \end{align*}
    and so taking the limit as $k\to\infty$ shows $\mu(\Bar{E}\cap B_\delta)=0$.
\end{proof}



\section*{Statements and Declarations}
\subsection*{Funding}
Consejo Nacional de Investigaciones Cient\'ificas y T\'ecnicas. 

\subsection*{Conflicts of interest/Competing interests}
The authors have no conflicts of interest to declare that are relevant to the content of this article.

\subsection*{Availability of data and material}
Not applicable.


\subsection*{Acknowledgements}
This work was supported by Consejo Nacional de Investigaciones Cient\'ificas y T\'ecnicas-CONICET and Universidad Nacional del Litoral-UNL, in Argentina.


\bigskip

\bigskip

%
\noindent{\textit{Affiliation.} 
	\textsc{Instituto de Matem\'{a}tica Aplicada del Litoral ``Dra. Eleonor Harboure'', CONICET, UNL.}

	\noindent \textit{Address.} \textmd{IMAL, Streets F.~Leloir and A.P.~Calder\'on, CCT CONICET Santa Fe, Predio ``Alberto Cassano'', Colectora Ruta Nac.~168 km~0, Paraje El Pozo, S3007ABA Santa Fe, Argentina.}

%
\noindent \textit{E-mail:} \verb|haimar@santafe-conicet.gov.ar| \\ \hspace*{1.3cm} \verb|ivanagomez@santafe-conicet.gov.ar| \\ \hspace*{1.3cm} \verb|ignaciogomez@santafe-conicet.gov.ar|
}
\end{document}